\documentclass[12pt,reqno]{amsart}
\usepackage{amsfonts,amssymb}
\usepackage[all]{xypic}
\usepackage{amsmath}
\usepackage{amsthm,amscd,latexsym}
\usepackage{color}

\newtheorem{theorem}{Theorem}[section]
\newtheorem{corollary}[theorem]{Corollary}
\newtheorem{Definition}[theorem]{Definition}
\newtheorem{lemma}[theorem]{Lemma}
\newtheorem{proposition}[theorem]{Proposition}
\newtheorem{Example}[theorem]{Example}
\newtheorem{Remark}[theorem]{Remark}

\newenvironment{remark}{\begin{Remark}\begin{em}}{\end{em}\end{Remark}}
\newenvironment{example}{\begin{Example}\begin{em}}{\end{em}\end{Example}}
\newenvironment{definition}{\begin{Definition}\begin{em}}{\end{em}\end{Definition}}

\newcommand{\Bx}{{\mathbf x}}
\newcommand{\By}{{\mathbf y}}

\newcommand{\Ba}{{\mathbf a}}
\newcommand{\Bb}{{\mathbf b}}
\newcommand{\un}{\small{1}}

\newcommand{\ve}{\varepsilon}

\DeclareMathOperator{\Herm}{\mathcal{H}}
\DeclareMathOperator{\ua}{\uparrow\!}

\DeclareMathOperator{\Pro}{\mathcal{P}}

\begin{document}
\title{Existence and Uniqueness of the Karcher Mean on Unital $C^*$-algebras}
\author[Lawson]{Jimmie Lawson}
\address{Department of Mathematics, Louisiana State University,
Baton Rouge, LA 70803, USA}\email{lawson@math.lsu.edu}
\date{December, 2018}

\maketitle

\noindent
\begin{abstract}
The Karcher mean on the cone $\Omega$ of invertible positive elements of
the $C^*$-algebra $\mathcal{B}(E)$ of bounded operators
on a Hilbert space $E$ has recently been extended to a contractive barycentric map
on the space of $L^1$- probability measures on $\Omega$. 
 In this paper we first show that the barycenter satisfies the Karcher equation and then establish 
 the uniqueness of the solution.  Next we establish
that the Karcher mean is real analytic in each of its coordinates, and use this fact to show that both
the Karcher mean and Karcher barycenter map exist and are unique on any unital $C^*$-algebra.
The proof depends crucially on a recent result of the author giving a converse of the inverse function theorem.

\end{abstract}



\section{Introduction}
In the past 15 years there has been a rather intensive study of geometric matrix means, in particular, one that  has been called
the Fr\'echet mean, Cartan mean, least squares mean, or Karcher mean, and their extensions to barycentric maps on spaces of
probability measures \cite{St}, \cite{LL16}.  H. Karcher  \cite{Ka77}
gave an equational characterization of this mean on Riemannian manifolds, and in \cite{LL14}  Y. Lim and the author showed that this equational characterization could
be extended to the cone of positive invertible operators in the $C^*$-algebra of bounded linear maps on a Hilbert space. Later M.\ Palfia studied
more general classes of means and barycenters and more general versions of the Karcher equation \cite{Pa}. 

The main purpose of this paper is to extend this theory of the Karcher mean to the cone of invertible positive elements in  \emph{any} unital $C^*$-algebra, since the methods of
\cite{LL14} do not extend beyond monotone complete $C^*$-algebras.  Indeed that paper closed with the problem of whether the theory could be extended 
to all unital $C^*$-algebras, and it is quite desirable to extend the theory to this general setting.
It should be remarked that this extension  has also recently been achieved by Y.\ Lim
and M. P\'alfia \cite{LP} using more technical machinery loosely based on well-known work of Crandell and Liggett, while the proof given here in some sense squeezes the result out of earlier results  on the Karcher mean.  In addition, each paper contains significant results not contained in
the other.

There are interesting side results that are derived in the course of the paper.   In Section 3 we derive strong connections between the barycentric map for
the arithmetic mean and the Bochner integral in the setting of probability measures. In Section 4 we note new and general continuity conditions holding for the 
contractive barycenter map extending the Karcher mean \cite{LL16} and use them to establish that the barycentric map yields solutions of the Karcher equation. 
 In Section 5 the uniqueness of the Karcher barycenter for integrable probability measures is established.  In Section 6 we show that the Karcher mean is 
 separately real analytic in each of its coordinates.  
In Section 7 we show using Section 6 and the Identity principle  that the cone $\Omega$ of positive invertible elements in any unital 
$C^*$-algebra is closed under the Karcher mean operation (in a completely different fashion from \cite{LP}).  The proof
 is tied up with showing via a recent converse of the inverse function theorem \cite{La18} 
 that Karcher mean as a function of one of its variables (with the others fixed) is a real analytic
diffeomorphism of $\Omega$.

\section{Preliminaries}
For a metric space $X$, let $\mathcal{B}(X)$ be the algebra of Borel sets, the smallest $\sigma$-algebra
containing the open sets.  A \emph{Borel measure} $\mu$ is a countably additive (positive) measure defined on $\mathcal{B}(X)$.
The support of $\mu$ consists of all points $x$ for which $\mu(U)>0$ for each open set $U$ containing $x$.  The support of
$\mu$ is always a closed set.  The \emph{finitely supported probability measures} are those of the form $\sum_{i=1}^n r_i\delta_{x_i}$, where
for each $i$, $r_i\geq 0$, $\sum_{i=1}^n r_i=1$, and $\delta_{x_i}$ is the point measure of mass $1$ at the point $x_i$.

We recall the \emph{Prohorov metric} $\pi(\mu,\nu)$
defined for two Borel probability measures $\mu,\nu$ on $X$ as the infimum of all $\ve>0$ such that for all closed sets $A$,
$$\mu(A)\leq \nu(A^\ve)+\ve,~~~\nu(A)\leq \mu(A^\ve)+\ve,$$
where $A^\ve=\{x\in X:d(x,y)<\ve \mbox{ for some }y\in A\}$.
The following result appears in \cite{La}.

\begin{proposition}\label{P:support1}
A  Borel probability measure $\mu$ on a metric space $(X,d)$ has separable support.  Furthermore, the following are equivalent.\\
(1) There exists a sequence $\{\mu_n\}$ of finitely supported measures $($with rational coefficients$)$ that converges to $\mu$ with respect to
the Prohorov metric.\\
(2) The support of $\mu$ has measure $1$, i.e., $\mu$ is
support-concentrated.
\end{proposition}

\begin{remark}\label{R:simple}
One may replace (1) by a sequence of uniform probability measures of finite support, i.e., of the form  $(1/n)\sum_{i=1}^n \delta_{x_i}$;
see Lemma 3.1 of \cite{LL16}.
\end{remark}

Let ${\mathcal P}(X)$ be the set of all support-concentrated Borel probability measures on $(X, {\mathcal B}(X))$
and  ${\mathcal P}_{0}(X)$ the set of all uniform probability measures $\mu=\frac{1}{n}\sum_{j=1}^{n}\delta_{x_{j}}$ of finite support.
Let ${\mathcal P}^{1}(X)\subseteq{\mathcal P}(X)$ be the set of probability measures with \emph{finite first moment}: for some (and hence all) $y\in X,$
$$\int_{X}d(x,y)d\mu(x)<\infty.$$
The Borel probability measures with finite first moment are also called the \emph{integrable} probability measures.

Let $(X,\mathcal{M})$ be a \emph{measure space}, a set $X$ equipped
with a $\sigma$-algebra $\mathcal{M}$, and $(Y,d)$ a metric space. A
function  $f:X\to Y$ is \emph{measurable}  if
$f^{-1}(A)\in\mathcal{M}$  whenever $A\in\mathcal{B}(Y)$. For $f$ to
be measurable, it suffices that $f^{-1}(U)\in\mathcal{M}$ for each
open subset $U$ of $Y$.  Hence  continuous functions are measurable
in the case $X$ is a metrizable space and
$\mathcal{M}=\mathcal{B}(X)$, the Borel algebra. A measurable map
$f:X\to Y$ between metric spaces  induces a \emph{push-forward} map
$f_*:\Pro(X)\to\Pro(Y)$ defined by $f_*(\mu)(B)=\mu(f^{-1}(B))$ for
$\mu\in\Pro(X)$ and $B\in\mathcal{B}(Y)$.  Note for $f$ continuous
that $\mathrm{supp}(f_*(\mu))=f(\mathrm{supp}(\mu))^-$, the closure
of the image of the support of $\mu$.

For $X$ a metric space, we say that $\omega\in\Pro(X\times X)$ is a \emph{coupling} for $\mu,\nu\in\Pro(X)$ and that $\mu,\nu$ are \emph{marginals} for
$\omega$ if for all $B\in\mathcal{B}(X)$
\begin{equation*}
\omega(B\times X)=\mu(B) \qquad \mathrm{and} \qquad \omega(X\times B)=\nu(B).
\end{equation*}
Equivalently $\mu$ and $\nu$ are the push-forwards of $\omega$ under the projection maps $\pi_1$ and $\pi_2$ resp.  We note that
one such coupling is the product measure $\mu\times \nu$, and that for any coupling $\omega$ it must be the case that
$\mathrm{supp}(\omega)\subseteq \mathrm{supp}(\mu)\times\mathrm{supp}(\nu)$.  We denote the set of all couplings for $\mu,\nu \in \Pro(X)$ by
$\Pi(\mu,\nu)$.

The  Wasserstein distance  $d_w$ (alternatively Kantorovich-Rubinstein distance)
on ${\mathcal P}^{1}(X)$ is defined by
\begin{equation}\label{E:winf}
d_w(\mu_{1},\mu_{2}):=\inf_{\pi\in\Pi(\mu_1,\mu_2)} \int_{X\times X} d(x,y) d\pi(x,y)..
\end{equation}
 It is known that $d_w$  is a metric  on ${\mathcal P}^{1}(X)$, is complete resp.\ separable whenever $d$ is complete resp.\ separable
 and that ${\mathcal P}_{0}(X)$ is $d_w$-dense in ${\mathcal P}^{1}(X)$ \cite{BO,St}.  
  
 For the case that $\mu=(1/n)\sum_{i=1}^n \delta_{x_i}$ and $\nu=(1/n)\sum_{i=1}^n\delta_{y_i}$, the equation (\ref{E:winf}) reduces to
\begin{equation}\label{E:winf2}
d_w(\mu,\nu)= \min_{\sigma\in S^{n}}\max\{d(x_{j},y_{\sigma(j)}): 1\leq j\leq n\},
\end{equation}
where $S^n$ is the permutation group on $\{1,\ldots,n\}$.

\section{Means and Barycenters}
We begin this section by recalling (Definition \ref{D:mb} through Proposition \ref{P:meanbary})  several needed notions and results from Section 3 of \cite{LL16}.

\begin{definition}\label{D:mb}
(1) An \emph{$n$-mean} $G_n$ on a set $X$ for $n\geq 2$
is a function $G_n:X^n\to X$ that is idempotent in the sense that $G_n(x,\ldots,x)=x$ for all  $x\in X$.
\newline
(2) An $n$-mean  $G_n$ is \emph{symmetric} or \emph{permutation invariant} if  for each permuation $\sigma$
of $\{1,\ldots, n\}$, $G_n(\Bx_\sigma) =G_{n}(\Bx)$, where $\Bx=(x_1,\ldots,x_n)$ and
${\Bx}_{\sigma}=(x_{\sigma(1)},\dots,x_{\sigma(n)})$. A (symmetric) \emph{mean} $G$ on $X$ is a sequence of means $\{G_n\}$,
one (symmetric) mean for each $n\geq 2$.
\newline
(3) A \emph{barycentric map} or \emph{barycenter} on the set of finitely supported uniform measures $\mathcal{P}_0(X)$ is a map
$\beta: \mathcal{P}_0(X)\to X$ satisfying $\beta(\delta_x)=x$ for each $x\in X$.
\end{definition}

For ${\bf x}=(x_{1},\dots,x_{n})\in X^{n},$ we let
\begin{equation}\label{E:1}
{\bf x}^{k}=(x_{1},\dots,x_{n}, x_{1},\dots,x_{n},\dots, x_{1},\dots,x_{n})\in X^{nk},
\end{equation}  where the number of ${\bf x}$-blocks is $k.$
We define the \emph{carrier} $S(\mathbf{x})$ of $\Bx$ to be the set of entries in $\mathbf{x}$, i.e.,
the smallest finite subset $F$ such that $\mathbf{x}\in F^n$.  We set $[\Bx]$ equal to the
equivalence class of all $n$-tuples obtained by permuting the coordinates of $\Bx=(x_1,\dots,x_n)$.  Note that the
operation $[\Bx]^k=[\Bx^k]$ is well-defined and that all members of $[\Bx]$ all have the same carrier set $S(\Bx)$.

A tuple $\Bx=(x_1,\ldots,x_n)\in X^n$ \emph{induces} on $S(\Bx)$ a uniform probability measure $\mu$
 with finite support by $\mu=\sum_{i=1}^n (1/n) \delta_{x_i}$, where $\delta_{x_i}$ is the point measure of mass 1 at $x_i$.
Since the tuple may contain repetitions of some of its entries, each singleton set
$\{x\}$ for $x\in\{x_1,\ldots,x_n\}$ will have measure $k/n$, where $k$ is the number
of times that it appears in the listing $x_1,\ldots, x_n$.  Note that every
member of $[\Bx]$ induces the same finitely supported probability measure.

\begin{lemma}\label{L:induced}
For each probability measure $\mu$ on $X$ with finite support $F$ for which $\mu(x)(=\mu(\{x\}))$ is rational for each $x\in F$,
there exists a unique $[\Bx]$ inducing $\mu$ such that any  $[\By]$ inducing $\mu$ is equal to $[\Bx]^k$ for some $k\geq 1$.
\end{lemma}

\begin{definition} A mean $G=\{G_{n}\}$ on $X$ is said to be \emph{iterative}  if for all  $n,k\geq 2$ and
all ${\bf x}=(x_{1},\dots,x_{n})\in X^{n}$,
$$G_{n}({\bf  x})=G_{nk}({\bf x}^{k}).$$
If a mean is both iterative and symmetric, we say it is \emph{intrinsic}.
\end{definition}
We have the following corollary to Lemma \ref{L:induced}.
\begin{corollary}\label{C:induced}
Let $G$ be an iterative mean.  Then for any finitely supported probability measure $\mu$ with support $F$ and taking on
rational values, we may define $\beta_G(\mu)=G_n(\Bx)$, for any $\Bx\in F^n$ that induces $\mu$.
\end{corollary}
Corollary \ref{C:induced} provides the basis for the following equivalence.
\begin{proposition}\label{P:meanbary}
There is a one-to-one correspondence between the intrinsic means and the barycentric maps on $\mathcal{P}_0(X)$ given
in one direction by assigning to an intrinsic mean $G$ the barycentric map $\beta_G$ and in the reverse
direction assigning to a barycentric map
$\beta$ the mean $G_n(x_1,\ldots, x_n)=\beta(\frac{1}{n}\sum_{i=1}^n\delta_{x_i})$.
\end{proposition}

We specialize to means and barycenters in metric spaces.
\begin{definition} \label{D:contr}
An $n$-mean $G_{n}:X^{n}\to X$ is said to be \emph{subadditive}
if for all ${\bf x}=(x_{1},\dots,x_{n}),{\bf
y}=(y_{1},\dots,y_{n})\in X^{n},$
\begin{eqnarray*}\label{E:Cmean}d(G_{n}({\bf x}), G_{n}({\bf y}))\leq
\frac{1}{n}\sum_{j=1}^{n}d(x_{j},y_{j}).
\end{eqnarray*}
%An $n$-mean is said to be \emph{submaxitive} if
%\begin{eqnarray*}
%d(G_{n}({\bf x}), G_{n}({\bf y}))\leq \mathrm{max}\{d(x_j,y_j) :\,
%1\leq j\leq n\}
%\end{eqnarray*}
%for all ${\bf x}=(x_{1},\dots,x_{n}),{\bf y}=(y_{1},\dots,y_{n})\in X^{n}.$
A mean $G=\{G_{n}\}$ is said to be subadditive if each $G_{n}$ is.
\end{definition}

We recall the definition of a contractive barycentric map.
\begin{definition} \label{D:cont}
Let $(X,d)$ be a metric space. A \emph{contractive barycentric map} on $\Pro^1(X)$ is a map
$\beta:{\mathcal P}^1(X)\to X$  satisfying $\beta(\delta_x)=x$ for all $x\in X$ and $d(\beta(\mu),\beta(\nu))\leq
d_w(\mu,\nu)$ for all $\mu,\nu\in\Pro^1(X)$.  \end{definition}

The following is part of Proposition 2.7 of \cite{LL16}.
\begin{proposition}\label{P:extend}
A subadditive intrinsic mean $G$  on a metric space $X$ uniquely
gives rise to a contractive barycentric map on $\Pro_0(X)$. If $X$
is complete, the barycentric map uniquely extends to a contractive
barycentric map $\beta_G:\Pro^1(X)\to X$ .
\end{proposition}

\begin{example}\label{Ex:Banach}
We consider the arithmetic mean $A=\{A_n\}$ on a Banach space $E$ defined by $A_n(x_1,\ldots,x_n)=(1/n)\sum_{i=1}^n x_i.$
It is easily checked to be symmetric and iterative, i.e., intrinsic.   To see that $A$ is contractive we compute
with the help of the triangle inequality for $\mathbf{x}=(x_1,\ldots, x_n)$ and $\mathbf{y}=(y_1,\ldots, y_n)$ in $E^n$,
$$ d(A_n(\mathbf{x}),A_n(\mathbf{y}))=\left\Vert \frac{1}{n}\sum_{i=1}^n x_i-\frac{1}{n}\sum_{i=1}^n y_i\right\Vert\leq\frac{1}{n}\sum_{i=1}^n \Vert x_i-y_i\Vert
=\frac{1}{n}\sum_{i=1}^n d(x_i,y_i).$$
Hence $A$ extends to a contractive barycentric map $\beta_E:\mathcal{P}^1(E)\to E$, called the \emph{arithmetic barycentric map} of $E$. We note
that if we weight $x_i$ with weight $w_i>0$ for $i=1,\ldots,n$, where $w_1+\ldots+w_n=1$, then, as expected, $\beta_E$ returns the weighted arithmetic mean
$\sum_{i=1}^n w_ix_i$.
\end{example}

As long as one is dealing with probability measures, there is a close relationship between the arithmetic barycentric map and the Bockner integral, to which
we now turn.   We consider functions defined on a probability measure space $(A,\mathcal{A},\mu)$, where $A$ is a set, $\mathcal{A}$ is a $\sigma$-algebra of
subsets of $A$ and $\mu$ is a $\sigma$-additive probability measure on $\mathcal{A}$. 
A $\mu$-\emph{simple function} with values in a Banach space $E$ 
is a function from $A$ into $E$ of the form $\sum_{i=1}^n \chi_{A_i}x_i$, where $A_i\in\mathcal{A}$, $x_i\in E$ and $\chi_{A_i}x_i(a)=x_i$ if $a\in A_i$ and $0$ otherwise. A function 
$f:A\to E$  is \emph{strongly $\mu$-measurable} if there exists a sequence $\{s_n\}$ of 
$\mu$-simple functions converging pointwise to $f$ $\mu$-almost everywhere.

\begin{definition} A function $f:A\to E$ is $\mu$-\emph{Bochner integrable} if there exists a 
sequence of $\mu$-simple functions $s_n :A\to E$ such that the following two conditions are met:\\
(1) $\lim_n s_n= f$  $\mu$-almost everywhere as $n\to\infty$;\\
(2) $\lim_n \int_A \Vert f-s_n\Vert\, d\mu=0$.
\end{definition}

We note that $f-s_n$ is strongly $\mu$-measurable since $f$ is by definition and $s_n$ is 
trivially, and hence $\Vert f-s_n\Vert$ is strongly $\mu$-measurable.  
Note that every $\mu$-simple function is $\mu$-Bochner integrable. For 
$s=\sum_{i=1}^n  \chi_{A_i}x_i$ we set 
$$\int_A s\,d\mu = \sum_{i=1}^n \mu(A_i)x_i.$$
It is routine to check that this definition is independent of of the representation of $s$.
If $f$ is $\mu$-Bochner integrable, the limit
$$\int_A f\, d\mu=\lim_{n\to\infty} \int_A s_n\, d\mu$$
exists in $E$ and is called the \emph{Bochner integral} of $f$  with respect to μ$\mu$.
 It is a standard result that this definition is independent of the approximating sequence 
 $\{s_n\}_{n=1}^\infty$ of $\mu$-simple functions.
 
 We recall that given $f:A\to E$ that is $\mu$-measurable, the \emph{pushforward measure} $f_*\mu$ is the measure
 on $E$ given by $f_*\mu(B)=\mu(f^{-1}(B))$ for any Borel set $B$.  The following proposition relates the Bockner
 integral to the arithmetic barycentric map $\beta_E$ of Example \ref{Ex:Banach}.
 \begin{proposition}\label{P:Bock}
 Let $(A,\mathcal{A},\mu)$ be a probabilistic measure space and $E$ a Banach space.  A measurable map $f:A\to E$ is
  Bochner integrable if and only if $f_*\mu\in \mathcal{P}^1(E)$ and in this case
 $\int_A f\,d\mu=\beta_E(f_*\mu)$. 
 \end{proposition}
 
 \begin{proof} A well-known necessary and sufficient condition for Bochner integrability is that $\int_A \Vert f\Vert\,d\mu<\infty$.
 By change of variables 
 $ \int_E d_E(x,0)\, df_*\mu(x)=\int_A \Vert f(a)\Vert d\mu(a)$, so $f*\mu$ is also integrable, i.e., is a member of $\mathcal{P}^1(E)$, if and
 only if $f$ is Bochner integrable.
 
Assuming $f$ is Bochner integrable, there exists a sequence of simple functions $s_n$ that converges 
pointwise $\mu$-almost everywhere to $f$ and for which 
 $$\int_A f\, d\mu=\lim_{n\to\infty} \int_A s_n\, d\mu.$$
Since $s_n$ is simple, the support of $(s_n)_*\mu$ is finite, and it is easily verified that $\int_A s_n\, d\mu$ is the weighted arithmetic mean
of  $(s_n)_*\mu$. Define $F:A\to E\times E$ by $F(a)=(f(a),s_n(a))$ and $\pi_n=F_*\mu$.  It is easy to verify that
$\pi_n$ has marginals $f_*\mu$ and  $(s_n)_*\mu$.  We calculate via change of variables, noting that the distance $d_E(x,y)=\Vert x-y\Vert$.  
$$ \int_{E\times E} \Vert x-y\Vert\, dF_*\mu(x,y)=\int_A  d_E(F(a))\, d\mu(a)=\int_A \Vert f-s_n\Vert\, d\mu.$$
Since $d_w(f_*\mu, (s_n)_*\mu)$ is bounded above by the preceding integral and the integral on the right approaches 
$0$ as $n\to 0$, we conclude that  $(s_n)_*\mu\to f_*\mu$  with respect to the Wasserstein distance.  
Since $\int_A s_n\, d\mu=\beta_E( (s_n)_*\mu)$ for each $n$, their limits $\int_A f \, d\mu$ and $\beta_E(f_*\mu)$ are equal.

 \end{proof}

\section{Solutions of the Karcher Equation:  Existence}
In this section we work in the following setting.  Let $\mathbb{A}$ be a $C^*$-algebra with identity and let 
$\Omega$ denote the open cone of positive invertible elements in $\mathbb{A}$.  Let $H=\Herm(\mathbb{A})$ denote
 the subspace of hermitian elements endowed with the norm metric and  topology.

The closed cone $\overline\Omega$ induces a closed partial order on $H$ defined by $x\leq y$ if $y-x\in\overline\Omega$.
We equip $\Omega$ with the Thompson metric, which can be defined from the partial order by
$$ d(x,y)=\log\max\{M(x/y),M(y/x)\} \mbox{ where } M(x/y)=\inf\{t>0:  x\leq ty\}.$$
The Thompson metric is a complete metric on $\Omega$ and the metric topology agrees with the relative norm topology.
The Thompson metric can be alternatively computed from the formula $d(x,y)=\Vert \log(x^{-1/2}yx^{1/2})\Vert$;
see \cite{LL14}.

The exponential function $\exp:H\to\Omega$ defined by the usual exponential power series is a diffeomorphism with inverse denoted by $\log$.
This exponential function is known to be expansive, so the logarithmic function is contractive.  For $x\in \Omega$, the map
$\theta_x(y)=x^{-1/2}yx^{1/2}$ is a linear order isomorphism on $H$ carrying $\Omega$ to $\Omega$, and restricted to $\Omega$
is an isometry for the Thompson metric (since it can be defined from the order relation).  Since $\theta_x$ is an isometry on $\Omega$
and $\log:\Omega\to H$ is contractive, they induce an isometry $(\theta_x)_*:\mathcal{P}^1(\Omega)\to\mathcal{P}^1(\Omega)$
and a contractive map $\log_*: \mathcal{P}^1(\Omega)\to \mathcal{P}^1(H)$, where the induced maps carry a measure to its
pushforward measure and the metrics are the appropriate Wasserstein metrics arising from the Thompson metric on $\Omega$ and the 
restricted norm metric on $H$.  

\begin{lemma}\label{L:cont1}
For $x\in \Omega$, the map $K_x: \mathcal{P}^1(\Omega)\to H$ defined by $$K_x(\mu)=\int_\Omega \log(x^{-1/2}ax^{-1/2})\, d\mu(a)$$ is contractive.
\end{lemma}

\begin{proof} From the previous paragraphs the composite map $\log_*(\theta_x)_*:\mathcal{P}^1(\Omega) \to \mathcal{P}^1(H)$ is contractive
with the image of $\mu$ being the pushforward measure under the composition $\log\circ\, \theta_x$.  Since the pushforward measure belongs
to $\mathcal{P}^1(H)$, we conclude from Proposition \ref{P:Bock} that the Bochner integral $\int_\Omega \log(x^{-1/2}ax^{-1/2})\, d\mu(a)$ exists
and is equal to $\beta_H(\log_*(\theta_x)_*(\mu))$.  Since the composition $\beta_H\log_*(\theta_x)_*$ is a composition of contractions,
it is a contraction and the lemma follows.
\end{proof}

\begin{lemma}\label{L:cont2}
The function  $K:\Omega\times \mathcal{P}^1(\Omega)\to H$ is continuous, where
$$ K(x,\mu)=\int_\Omega \log(x^{-1/2}ax^{-1/2})\, d\mu(a)$$
\end{lemma}

\begin{proof}
Choose convergent sequences $x_n\to x$ in $\Omega$ and $\mu_n\to \mu$ in $\mathcal{P}^1(\Omega)$.  There exists $M$ such that
$d(x_n,x)\leq M$ for all $n$.  Define $f_n, f: \Omega\to H$ by $f_n(a)=\log(x_n^{-1/2}ax_n^{-1/2})$ and $f(a)=\log(x^{-1/2}ax^{-1/2})$.
By continuity of multiplication $f_n$ converges pointwise to $f$  and for all $n$
\begin{eqnarray*}
\int_\Omega\Vert f_n(a)\Vert\,d\mu(a) &=& \int_\Omega\Vert \log(x_n^{-1/2}ax_n^{-1/2})\Vert d\mu(a)\\
&=&\int_\Omega d(x_n,a)\,d\mu(a)\leq \int_\Omega (M+d(x,a))\,d\mu(a)\\
&\leq& M+\int_\Omega d(x,a)\,d\mu(a)<\infty.
\end{eqnarray*}
where the last inequality follows from $\mu\in\mathcal{P}^1(\Omega)$.  Thus the family $\{\Vert f_n\Vert\} $ is uniformly bounded by 
an integrable function, and hence by Lebesgue Bounded Convergence Theorem for the Bochner integral, we conclude 
that $\lim_n \int_\Omega f_n d\mu=\int_\Omega f\,d\mu$.

Let $\ve>0$ and pick $N$ such that $n\geq N$ implies $\Vert \int_\Omega f\, d\mu-\int_\Omega f_n\,d\mu\Vert<\ve/2$ and
$d_w(\mu,\mu_n)<\ve/2$.   We then have for $n\geq N$
\begin{eqnarray*}
\Vert K(x,\mu)-K(x_n,\mu_n)\Vert& = &\Big\Vert \int_\Omega f\,d\mu-\int_\Omega f_n\,d\mu_n
\Big\Vert\\
&\leq& \Big\Vert \int_\Omega (f-f_n)\, d\mu\Big\Vert +\Big\Vert \int_\Omega f_n\, d\mu-\int_\Omega f_n\, d\mu_n\Big\Vert\\
&\leq& \ve/2+d_w(\mu,\mu_n)<\ve,
\end{eqnarray*}
where we apply Lemma \ref{L:cont1} to obtain the first inequality in the last line.
\end{proof}

We say $x\in\Omega$ \emph{is a solution of the Karcher equation for} $\mu\in\mathcal{P}^1(\Omega)$,
or alternatively that $(x,\mu)$ \emph{is  a solution of the Karcher equation},
if \begin{equation}\label{E:karch1}
\int_\Omega \log(x^{-1/2}ax^{-1/2})\,\mu(a)=0.\end{equation}
  For the case that $\mu=\sum_{i=1}^n \omega_i\delta_{a_i}$ is a finitely
supported probability measure the  Karcher equation becomes
\begin{equation}\label{E:karch2}
\sum_{i=1}^n \omega_i\log (x^{-1/2}a_ix^{-1/2})=0.
\end{equation}

For the special case of the $C^*$-algebra $\mathbb{B}(E)$ of bounded linear operators on a Hilbert
space $E$, it was shown in \cite[Theorem 6.5]{LL14} (in the setting of means) that any finitely
supported probability measure has  a unique solution $\Lambda(\mu)$ for its Karcher equation,
and that solution is called the \emph{Karcher mean}.  (There is an alternative characterization
given in \cite{LL14} of the Karcher mean as a limit of power means.)

In \cite{LL16} the Karcher mean was extended from $\mathcal{P}_0(\Omega)$ to 
$\mathcal{P}^1(\Omega)$ by observing that $\Lambda:\mathcal{P}_0(\Omega)\to 
\Omega$ is contractive from the Wasserstein metric to the Thompson metric, and since
$\mathcal{P}_0(\Omega)$ is dense in $\mathcal{P}^1(\Omega)$ and the Thompson metric
is complete \cite{Th}, $\Gamma$ extends to a contractive  barycentric map from
$\mathcal{P}^1(\Omega)$ to $\Omega$.  It makes sense to call this extension the 
\emph{Karcher barycentric map}.  We can use Lemma \ref{L:cont2} to show the 
desirable result that this extension still satisfies the Karcher equation.

\begin{proposition}\label{P:Kexists}
Let $\Omega$ be the open cone (in the subspace $\mathcal{H}(E)$ of hermitian operators)
of positive invertible elements, let $\Lambda:\mathcal{P}^1(\Omega)\to\Omega$ be the Karcher barycenter map,
and let $\mu\in \mathcal{P}^1(\Omega)$. Then $x=\Gamma(\mu)$ satisfies the Karcher equation (\ref{E:karch1}).
\end{proposition}

\begin{proof}
Let $\{\mu_n\}$ be a sequence in $\Pro_0(\Omega)$ converging to $\mu$, and let $\Lambda(\mu_n)$ be the
Karcher mean for each $\mu_n$.  By continuity of $\Lambda$, $\lim_n\Lambda(\mu_n)=\Lambda(\mu)$.
By Lemma \ref{L:cont2}  
$$K(\Lambda(\mu), \mu)=\lim_n K(\Lambda(\mu_n),\mu_n) =0,$$
which completes the proof.
\end{proof}

\section{Solutions of the Karcher Equation: Uniqueness}
In this section we work in the setting of the $C^*$-algebra $\mathbb{A}=\mathcal{B}(E)$ of bounded linear operators on a Hilbert space $E$
with $\Omega$ denoting the cone of invertible positive operators and $H=\mathcal{H}(E)$ the Banach space of hermitian operators.

\begin{lemma}\label{L:support}
Let $(X,d)$ be a metric space, let $\mu,\nu\in \mathcal{P}(X)$, and suppose $B_r(z)\cap\mathrm{supp\,}\mu\ne\emptyset$, where $B_r(z)$ is the
open ball of radius $r$ around $z$.  If $B_{2r}(z)\cap\mathrm{supp}\,\nu=\emptyset$, then $d_w(\mu,\nu)\geq r\mu(B_r(z))$.
\end{lemma}

\begin{proof} Since $B=B_r(z)$ meets $\mathrm{supp}\,\mu$, $\mu(B)>0$. 
For any $\pi\in \Pi(\mu,\nu)$, supp$(\pi)\subseteq \mathrm{supp}(\mu)\times\mathrm{supp}(\nu)$, so
\begin{eqnarray*}
\int_{X\times X} d(x,y)d\pi(x,y) &=& \int_{X\times \mathrm{supp}\,\nu} d(x,y) d\pi \\
&=&  \int_{B\times \mathrm{supp}\,\nu} d(x,y) d\pi+\int_{X\setminus B\times \mathrm{supp}\,\nu} d(x,y) d\pi\\
&\geq& \int_{B\times \mathrm{supp}\,\nu} d(x,y) d\pi\geq \int_{B\times \mathrm{supp}\,\nu} r\, d\pi\\
&=& r\int_{B\times \mathrm{supp}\,\nu}  d\pi =r\pi(B\times X)=r\mu(B).
\end{eqnarray*}
Since $\pi\in\Pi(\mu,\nu)$ was arbitrary, the lemma follows.
\end{proof}

%In what follows we work in the following setting.  Let $\mathbb{A}$ be a $C^*$-algebra with identity and let 
%$\Omega$ denote the open cone of positive invertible elements in $\mathbb{A}$ equipped with the Thompson
%metric $d$.  Let $H=\Herm(\mathbb{A})$ denote the subspace of hermitian elements endowed with the norm metric and 
%topology.

\begin{lemma}\label{L:approx}
If $(x,\mu)$ is a solution of the Karcher equation and $\ve>0$, then there exists $\eta \in \mathcal{P}_0(\Omega)$ such that
$d_w(\mu,\eta)<\ve$ and $(x,\eta)$ is a solution of the Karcher equation.
\end{lemma}

\begin{proof}  Since $\Pro_0(\Omega)$ is dense in $\Pro^1(\Omega)$, we may select a sequence $\{\mu_n\}\subseteq\Pro_0(\Omega)$
that converges to $\mu$ with respect to the Wasserstein metric on $\Pro^1(\Omega)$ arising from the Thompson metric on $\Omega$.
Without loss of generality we may assume that $\mu_n=(1/N)\sum_{i=1}^N \delta_{x_i}$  for some $N\geq n$, since otherwise we may uniformly
divide each mass $(1/N)\delta_{x_i}$ into $m$ equal parts, where $mN>n$ and rewrite the summation accordingly. Pick $R$ large enough so
that the open ball $B=B_R(1)$ of radius $R$ around $1$ meets supp$(\mu)$; thus $\mu(B)>0$.  For all sufficiently large $n$, 
$d_w(\mu,\mu_n)<R\mu(B)$, so by Lemma \ref{L:support}, supp$(\mu_n)\cap B_{2R}(1)\ne\emptyset$ for such $n$.  We can thus relabel if necessary so  that for $\mu_n=(1/N)\sum_{i=1}^N \delta_{x_i}$,
we have $d(1,x_1)<2R$.  

By Lemma \ref{L:cont1} the map $K_x(\nu)=\int_\Omega x^{-1/2}ax^{-1/2}d\nu(a)$  from
$\mathcal{P}^1(\Omega)$ to $H$ is contractive. Since by hypothesis $K_x(\mu)=0$
we have $$0=\lim_{n\to\infty}K_x(\mu_n)=\lim_n\frac{1}{N}\sum_{i=1}^N 
\log(x^{-1/2}x_ix^{-1/2}),$$ 
where $N$ and each $x_i$ are functions of $n$.  We note that $\lim_n (1/N)\log(x^{-1/2}x_1x^{-1/2})=0$ 
since $N\geq n$ and
\begin{eqnarray*}
d_{\Vert\cdot\Vert}(0,\log(x^{-1/2}x_1x^{-1/2})&=&d_{\Vert\cdot\Vert}(\log(1),\log(x^{-1/2}x_1x^{-1/2})\leq d(1,x^{-1/2}x_1x^{-1/2})\\
&\leq & d(1,x^{-1})+d(x^{-1/2}(1)x^{-1/2},x^{-1/2}x_1x^{-1/2})\\
&=&d(1,x^{-1})+d(1,x_1)\leq d(1,x^{-1})+2R.
\end{eqnarray*}
The bottom equality follows since the map $x\mapsto axa$ on $\Omega$ is
an isometry.  We conclude that
\begin{eqnarray}\label{E:lim0}
0&=&\lim_n \frac{1}{N}\sum_{i=1}^N \log(x^{-1/2}x_ix^{-1/2})\nonumber\\
&=&\lim_n \Big[\frac{1}{N} \log(x^{-1/2}x_1x^{-1/2})+\frac{1}{N}\sum_{i=2}^N \log(x^{-1/2}x_ix^{-1/2})\Big] \nonumber\\
&=&\lim_n \frac{1}{N}\sum_{i=2}^N \log(x^{-1/2}x_ix^{-1/2}).
\end{eqnarray}
We set $y_1=x^{1/2}\exp(-\sum_{i=2}^N \log(x^{-1/2}x_ix^{-1/2}))x^{1/2}$ and define 
$\nu_n\in\Pro_0(\Omega)$ by 
$\nu_n=(1/N)(\delta_{y_1}+\sum_{i=2}^N \delta_{x_i})$.  It then follows
directly from the definition of $y_1$ that
$$\frac{1}{N} \log(x^{-1/2}y_1x^{-1/2})+\frac{1}{N}\sum_{i=2}^N \log(x^{-1/2}x_ix^{-1/2})=0,$$
and hence that $(x,\nu_n)$ is a solution of the Karcher equation.

To complete the proof we show $d_w(\mu_n,\nu_n)\to 0$ as $n\to\infty$, and hence that 
$\lim_n\nu_n=\mu$.  To obtain $\nu_n$ to $\mu_n$, we simply transport the mass of
$(1/N)$ from $x_1$ to $y_1$.  Hence the Wasserstein distance $d_w(\mu_n,\nu_n)$ 
satisfies $d_w(\mu_n,\nu_n)\leq (1/N)d_w(x_1,y_1)$.   We note that
\begin{eqnarray*}
d_w(x_1,y_1)&\leq&d(x_1,1)+d(1,x^{1/2}\exp(-\sum_{i=2}^N \log(x^{-1/2}x_ix^{-1/2})x^{1/2})\\
&\leq& 2R+ d(1,x)+d(x,x^{1/2}\exp(-\sum_{i=2}^N \log(x^{-1/2}x_ix^{-1/2})x^{1/2})\\
&=& 2R+d(1,x)+d(1,\exp(-\sum_{i=2}^N \log(x^{-1/2}x_ix^{-1/2})))\\
&=&2R+d(1,x)+\Big\Vert \sum_{i=2}^N \log(x^{-1/2}x_ix^{-1/2})\Big\Vert.
\end{eqnarray*}
Thus 
$$d_w(\mu_n,\nu_n)\leq \frac{d(x_1,y_1)}{N}\leq \frac{2R+d(1,x)+\Big\Vert \sum_{i=2}^N \log(x^{-1/2}x_ix^{-1/2})\Big\Vert}{N},$$
and we conclude from this inequality and (\ref{E:lim0}) that $\lim_n d_w(\mu_n,\nu_n)=0$.
\end{proof}

\begin{theorem}\label{T:uniq}
Let $\Omega$ be the open cone (in the subspace $\mathcal{H}(E)$ of hermitian operators)
of positive invertible elements, let $\Lambda:\mathcal{P}^1(\Omega)\to\Omega$ be the Karcher barycenter map,
and let $\mu\in \mathcal{P}^1(\Omega)$. Then $x=\Gamma(\mu)$ is the unique solution of  the Karcher equation (\ref{E:karch1})
for $\mu\in \Pro^1(\Omega)$.  
\end{theorem}  

\begin{proof}  Suppose that $(x,\mu)$ and $(y,\mu)$ are both solutions, $x\ne y$.  Set $\ve=(1/2)d(x,y)$.  By Lemma \ref{L:approx}
we can pick $\eta,\tau\in\Pro_0(\Omega)$ such that $d(\eta,\mu)<\ve$, $d(\tau,\mu)<0$ and $(x,\eta)$ and $(y,\tau)$
are solutions of the Karcher equation.  By uniqueness of the Karcher mean for members of $\Pro_0(\Omega)$ \cite{LL14}, $x=\Lambda(\eta)$
and $y=\Lambda(\tau)$.  Since $\Lambda$ is contractive \cite{LL16}, 
$$d(x,y)=d(\Lambda(\eta),\Lambda(\tau))\leq d_w(\eta,\tau)\leq d_w(\eta,\mu)+d_w(\mu,\tau)< 2\ve=d(x,y),$$
a contradiction.
\end{proof}

\section{Analyticity}
In this section we continue to work in the setting of $\mathbb{A}=\mathcal{B}(E)$, the $C^*$-algebra of bounded linear operators on a Hilbert space $E$ with cone $\Omega$ of positive  invertible elements.
Let $\mathcal{H}(\mathbb{A})$  denote the closed subspace of hermitian operators and let $\un$ denote the identity of $\mathbb{A}$.

It is standard that the maps $\exp:\mathcal{H}(\mathbb{A})\to\Omega$ and $\log:\Omega\to 
\mathcal{H}(\mathbb{A})$ are inverse real analytic diffeomorphisms .  Moreover, they are regular in the sense that their derivatives at each point in the domain are invertible maps 
(see, for example,  the proof of Theorem V.1 and Examples VI.6 and VI.7 of \cite{Neeb} or Theorem 9.1 and its proof in \cite{LL07}).

As mentioned earlier, for $a_1,\ldots,a_n,x\in \Omega$ the Karcher equation 
\begin{equation}\label{E:KM0} \log(x^{-1/2}a_1x^{-1/2})+\cdots+ \log(x^{-1/2}a_n x^{-1/2})=0.\end{equation}
has a unique solution $x=\Lambda(a_1,\ldots,a_n)\in\Omega$, the Karcher mean.

\begin{lemma} \label{L:KM1}
For fixed $a_1,\ldots,a_{n-1},x\in\Omega$, there uniquely exists $a\in \Omega$ such that 
\begin{equation}\label{E:KM2} \log(x^{-1/2}a_1x^{-1/2})+\cdots+\log(x^{-1/2}a_{n-1} x^{-1/2})+ \log(x^{-1/2}a x^{-1/2})=0.\end{equation}
The solution is given by $a=x^{1/2}\exp\big(-\sum_{i=1}^{n-1}\log(x^{-1/2}a_ix^{-1/2})\big)x^{1/2}$.
The solution $a$ is also uniquely determined by $x=\Lambda(a_1,\ldots, a_{n-1},a)$.
\end{lemma}
\begin{proof}  One verifies directly by substitution that the $a$ defined in the lemma is indeed a solution
to equation (\ref{E:KM2}). If $b$ were another solution, then it would follow that  $\log(x^{-1/2}a x^{-1/2})=\log(x^{-1/2}b x^{-1/2})$
and hence $a=b$.  The fact that equation (\ref{E:KM2}) holds and the uniqueness of the Karcher mean yield the last assertion.
\end{proof}

\begin{definition}\label{D:KM3}
For fixed $\mathbf{a}=(a_1,\ldots,a_{n-1}) $, the preceding lemma gives rise to a real analytic function
$\gamma:\Omega\to \Omega$ defined by 
$$\gamma(x)=\gamma_{\mathbf{a}}(x)=x^{1/2}\exp\big(-\sum_{i=1}^{n-1}\log(x^{-1/2}a_ix^{-1/2})\big)x^{1/2}.$$
 We also define $\lambda:\Omega\to\Omega$ by $\lambda(a)=\Lambda(a_1,\ldots ,a_{n-1},a)$.  
 \end{definition}
 
 \begin{lemma}\label{L:KM4}
 For fixed $a_1,\ldots, a_{n-1}\in \Omega$, $\gamma$ and $\lambda$ are inverse functions.
 \end{lemma}
 
 \begin{proof}
 That each of  $\lambda\gamma$ and $\gamma\lambda$ is the identity on $\Omega$ follows directly from Lemma \ref{L:KM1}.  
\end{proof}

\begin{lemma}\label{L:lips}
The identity function from $(\Omega,\Vert\cdot\Vert)$  to $(\Omega,d)$ is locally Lipschitz, and hence $\lambda:(\Omega, \Vert\cdot
\Vert)\to (\Omega,d)$ is locally Lipschitz.
\end{lemma}

\begin{proof}
Let $W_M=\{x\in\Omega: Mx\in \un+ \Omega\}$.  Then $W_M$ is an open set, and $\bigcup_M W_M=\Omega$.  
For $x,y\in W_M$,  let $\beta=\Vert x-y\Vert$.  Since $x-y\leq \Vert x-y\Vert\un$, we obtain
\begin{eqnarray*}
x&\leq& y+\Vert x-y\Vert\un\leq y+M\Vert x-y\Vert y=(1+M \Vert x-y\Vert)y
\end{eqnarray*}
and interchanging the role of $x$ and $y$ yields  $y\leq (1+M \Vert x-y\Vert)x$.
It follows that $$d(x,y)\leq \log (1+M \Vert x-y\Vert)\leq  M\Vert x-y\Vert$$
since $\log(1+Mt)\leq Mt$ for $t\geq 0$.

For the second assertion the contractivity of $\Lambda$ \cite{LL16} implies $\lambda:(\Omega, d)\to (\Omega,d)$ is
contractive.  So $\lambda:(\Omega, \Vert\cdot\Vert)\to (\Omega,d)$, which is the composition of this map and that of the
first paragraph, is the composition of a contractive map and a locally Lipschitz map, hence locally Lipschitz.
\end{proof}

The next recent result, a converse of the inverse function theorem and the main result of \cite{La18}, is crucial for what follows.
\begin{theorem}\label{T:conv}
 Let $g:U\to V$ and 
$f:V\to U$ be inverse homeomorphisms between open subsets of Banach spaces.  If $g$ is differentiable of class
$C^p$ and $f$ if locally Lipschitz, then the Fr\'echet derivative of $g$ at each point of $U$ is invertible and
$f$ must be differentiable of class $C^p$.
\end{theorem}

\begin{corollary}\label{C:big}
Let $\Omega$ be the cone of invertible positive elements in the $C^*$-algebra $\mathcal{B}(E)$ of bounded operators
on a Hilbert space $E$. For fixed $a_1,\ldots, a_{n-1}\in \Omega$ the function $\lambda(a)=\Lambda(a_1,\ldots, a_{n-1},a)$
is an analytic diffeomorphism on $\Omega$.
\end{corollary}

\begin{proof} By Lemma \ref{L:KM4}, $\lambda$ and $\gamma$ or inverse functions and by Lemma \ref{L:lips} $\lambda$
is locally Lipschitz.  From Definition \ref{D:KM3} $\gamma$ is a composition of real analytic functions, hence real analytic.
Thus the corollary follows from the preceding theorem for the case $p=\omega$, the real analytic case.
\end{proof}

\section{The Karcher mean in general $C^*$-algebras}
In this section we extend previous results to the setting of a general unital $C^*$-algebra $\mathbb{A}$.
We begin with what is frequently called the Identity Theorem and at other times grouped with the
principle of analytic continuation. For a proof of the version we quote, see \cite[Theorems 1.11, 3.1]{Up}.
\begin{theorem}\label{T:AC}
Suppose that $D$ is an open connected subset of a Banach space $E$ and that $f,g:D\to F$ are real analytic maps
into another Banach space.  If there exists an open subset $U$ of $D$ such that $f\vert_U=g\vert_U$, then $f=g$.
\end{theorem}

The next theorem is proved as part of Theorem 6.3 of \cite{LL14}.  The proof there is given in the context of a 
$C^*$-algebra of bounded operators on a Hilbert space, but works just as well in the context of a general unital
$C^*$-algebra.  Also we use the modification of the implicit function theorem used in the proof of that case to
the real analytic case.
\begin{theorem}\label{T:local} Let $\mathbb{A}$ be a unital $C^*$-algebra.   Then there exists an open set $U$ containing
the identity $\un$ such that  the restriction of the Karcher mean $\Lambda: U^n\to \Omega$ is real analytic.
\end{theorem}

By the Gelfand-Naimark theorem the $C^*$-algebra $\mathbb{A}$ is isometrically $*$-isomorphic to a closed $C^*$-subalgebra 
containing the identity of a $C^*$-algebra $\mathcal{B}(E)$ of bounded operators on a Hilbert space. We fix some such
$\mathcal{B}(E)$ and hence forth identify $\mathbb{A}$ with the isomorphic subalgebra.  

\begin{lemma}\label{L:closure}
 For any $a_1,\ldots, a_n\in \Omega$, the cone of invertible positive elements of $\mathbb{A}$, we have
 $\Lambda(a_1,\ldots, a_n)\in \Omega$.
 \end{lemma}

\begin{proof}  We consider the map $\Lambda$ defined on the cone of positive elements in $\mathcal{B}(E)$ restricted to
$\Omega$,  $\Lambda: \Omega^n\to \mathcal{B}(E)$.  Suppose that $a_1,\ldots, a_{n-1}\in U$, where $U$ is given in
Theorem \ref{T:local}, $b\in\Omega$, but $z=\Lambda(a_1,\ldots, a_{n-1},b)\notin \Omega$,
hence not in $\mathbb{A}$ (since an element of $\mathbb{A}$ is a positive invertible  
in $\mathbb{A}$ if and only if is in $\mathcal{B}(E)$).  By the Hahn-Banach theorem there exists a continuous linear functional
$\pi$ defined on $\mathcal{B}(E)$ such that $\pi(\mathbb{A})=0$ and $\pi(z)\ne 0$.  By Lemma \ref{L:closure}, however, 
$\pi(\Lambda(a_1,\ldots, a_{n-1},a)=0$ for all $a\in U$. Thus the composition $a\mapsto \pi(\Lambda(a_1,\ldots, a_{n-1},a)$
 is a real analytic function on $\Omega$ which is $0$ on $U$. Hence by Theorem \ref{T:AC} it agrees with the $0$ function everywhere,
 which contradicts $\pi(\Lambda(a_1,\ldots, a_{n-1},b))\ne 0$.
 
 We next consider the case $a_1,\ldots, a_{n-2}\in U$, $a_n\in\Omega$, and $\Lambda(a_1,\ldots,a_{n-2},b, a_n)\notin \Omega$.
 We proceed along the lines of the preceding paragraph, noting from the previous paragraph that 
 $\Lambda(a_1,\ldots,a_{n-2},a, a_n)\in \Omega$ for all $a\in U$, and again derive a contradiction.  We can then proceed inductively
 through all the coordinates, and derive the assertion of the lemma.
\end{proof}

With the aid of Lemma \ref{L:closure} we can restrict the results of the earlier section to $\mathbb{A}$ and derive existence and
uniqueness of the Karcher mean in an arbitrary unital $C^*$-algebra. 
\begin{theorem}\label{T:gencstar}
In any unital $C^*$-algebra $\mathbb{A}$ the Karcher mean and the Karcher barycenter maps are defined on $\Omega$,
the cone of positive invertible elements.  In each case the Karcher mean or barycenter is the unique solution of the corresponding
Karcher equation.
\end{theorem}

\begin{proof} As remarked before the statement of the theorem, the Karcher mean case follows from our earlier results.
The means may be reinterpreted as a contractive barycenter map on $\Pro_0(\Omega)$ (as in \cite{LL16}), which extends
uniquely to a contractive barycenter map on $\Pro^1(\Omega)$. By Proposition \ref{P:Kexists} the contractive barycentric map
sends a member of $\Pro^1(\Omega)$ to its Karcher barycenter.  Since the image of $\Pro_0(\Omega)$ is dense in the image
of $\Pro^1(\Omega)$, the latter image will be contained in the positive invertible elements of $\mathbb{A}$, that is, in $\Omega$.
The Karcher equation of $\mathbb{A}$ is the restriction of the one for $\mathcal{B}(E)$, so existence and uniqueness follow 
from the earlier corresponding results for $\mathcal{B}(E)$.

\end{proof}

These results on the existence and uniqueness of the Karcher mean have been obtained  earlier by Lim and Palfia \cite{LP} by entirely 
different methods.   They show based on Crandall-Liggett type arguments that the Karcher mean is the norm limit of the power means, 
which establishes the existence of the Karcher mean in an arbitrary unital $C^*$-algebra.  The author views the approach given here as
more elementary and computationally less intense.

We also have the following variant of Corollary \ref{C:big}.
\begin{corollary}\label{C:big2}
Let $\Omega$ be the cone of invertible positive elements in a unital  $C^*$-algebra $\mathbb{A}$.
For fixed $a_1,\ldots, a_{n-1}\in \Omega$ the function $\lambda(a)=\Lambda(a_1,\ldots, a_{n-1},a)$
is an analytic diffeomorphism on $\Omega$.
\end{corollary}

\begin{proof} As previously we assume by the Gelfand-Naimark theorem the $C^*$-algebra $\mathbb{A}$ is 
isometrically $*$-isomorphic to a closed $C^*$-subalgebra 
containing the identity of a $C^*$-algebra $\mathcal{B}(E)$ of bounded operators on a Hilbert space.
Directly from its definition (Definition \ref{D:KM3}) $\gamma$ carries $\Omega$ into $\Omega$ and it follows
from Lemma \ref{L:closure} it follows that $\lambda$ also carries $\Omega$ into $\Omega$.  From
Lemma \ref{L:KM4} $\gamma$ and $\lambda$ are inverse maps on the cone of positive invertible operators
in $\mathcal{B}(E)$, the same holds for their restrictions to $\Gamma$.  The analyticity now follows from
Corollary \ref{C:big}.
\end{proof}

\section{Properties of the Karcher mean and barycenter in general $C^*$-algebras}
In this section we work in the general setting of a unital $C^*$-algebra $\mathbb{A}$.  We denote the closed subspace 
$\mathcal{H}(\mathbb{A})$ of hermitian elements by $H$, and the cone of positive invertible elements by $\Omega$, an open
cone in $H$.  For  a \emph{weight} $\omega=(w_{1},\dots,w_{n})$ with entries non-negative real numbers summing to $1$ and 
 $\Ba=(a_1,\ldots, a_n)\in\Omega^n$, we consider the  corresponding Karcher equation on $\Omega$
\begin{equation}\label{E:least1}
\sum_{i=1}^{n}w_{i}\log(x^{1/2}a_{i}^{-1}x^{1/2})=0,\end{equation}
with solution the Karcher mean $x=\Lambda(\omega, \Ba)$.  By Theorem \ref{T:gencstar} the Karcher equation has a unique solution in $\Omega$.  

In \cite[Theorem 6.8]{LL14} several basic properties of the Karcher mean were derived in the setting of $\mathbb{A}=\mathcal{B}(E)$,
 the $C^*$-algebra of bounded linear operators on a Hilbert space $E$. 
These properties can be carried over to a general unital $C^*$-algebra $\mathbb A$  by first 
embedding $\mathbb{A}$ into some $\mathcal{B}(E)$ as a closed $C^*$-subalgebra (by the
 Gelfand-Naimark theorem), noting that the Karcher mean $\Lambda$ of $\mathcal{B}(E)$ restricted to the embedded $\mathbb{A}$
 is the Karcher mean of $\mathbb{A}$, and hence the equalities and inequalities of the following theorem that hold
 in $\mathcal{B}(E)$ remain valid in $\mathbb{A}$.

\begin{theorem}\label{T:props}
For a weight $\omega=(w_1,...w_n)$  and $\Ba=
(a_1,\ldots,a_n)\in \Omega^n$, the following properties hold.
\begin{itemize}
\item[(P1)] $($Consistency with scalars$)$
$ \Lambda(\omega;\Ba)=a_{1}^{w_{1}}\cdots a_{n}^{w_{n}}$ if the
$a_{i}$'s commute;
\item[(P2)] $($Joint homogeneity$)$
$ \Lambda(\omega; (t_{1}a_{1}, \dots, t_{n}a_{n}))= t_{1}^{w_{1}}\cdots
t_{n}^{w_{n}}\Lambda(\omega;\Ba)$, all $t_i>0$;
\item[(P3)] $($Permutation invariance$)$
$ \Lambda(\omega_{\sigma};\Ba_{\sigma}) =\Lambda(\omega;\Ba)$,
where $\omega_{\sigma}=(w_{\sigma(1)},\dots,w_{\sigma(n)});$
\item[(P4)] $($Monotonicity$)$ If $b_{i}\leq a_{i}$ for all $1\leq i\leq n,$ then
$\Lambda(\omega;\Bb)\leq \Lambda(\omega;\Ba);$
\item[(P5)] $($Continuity$)$ $\Lambda$ is
contractive for the Thompson metric$;$
\item[(P6)] $($Congruence invariance$)$
$\Lambda(\omega;m^{*}\Ba m)= m^{*}\Lambda(\omega;\Ba) m$ for
any invertible $m;$
\item[(P7)] $($Joint concavity$)$ $\Lambda(\omega;\lambda \Ba+(1-\lambda)\Bb)\geq \lambda \Lambda(\omega;\Ba)+(1-\lambda)\Lambda(\omega;\Bb)$ for\, $0\leq \lambda \leq 1$;
\item[(P8)] $($Self-duality$)$
$ \Lambda(\omega;a_{1}^{-1} ,\dots, a_{n}^{-1})^{-1}=
\Lambda(\omega;a_{1} ,\dots, a_{n});$ and
\item[(P9)] $($AGH weighted  mean inequalities$)$
$ (\sum_{i=1}^{n}w_{i}a_{i}^{-1})^{-1}\leq \Lambda(\omega;\Ba)
\leq \sum_{i=1}^{n}w_{i}a_{i}.$
\end{itemize}
\end{theorem}

Some the these properties extend naturally to the Karcher barycenter map $\Lambda:\Pro^1(\Omega)\to \Omega$.  The following items are 
key observations for these extensions.

\begin{definition}
A Lipschitz function $F:\Omega\to \Omega$ \emph{preserves the Karcher barycenter} $\Lambda: \Pro^1(\Omega)\to \Omega$ if 
for all $\mu\in\Pro^{1}(\Omega)$, $F(\Lambda(\mu))=\Lambda(F_*(\mu)$, where $F_*(\mu)$ is the pushforward of $\mu$.  
Similarly $F$ \emph{preserves the Karcher mean}  if for any $n\geq 2$, $x_1,\ldots, x_n\in \Omega$,
$F(\Lambda(x_1,\ldots,x_n))=\Lambda(F(x_1),\ldots,F(x_n))$.
\end{definition}

\begin{remark}\label{R:pres}
It is easy to verify that $F$ preserves the Karcher mean if and only if $F$ preserves the Karcher barycenter for all $\mu\in\Pro_0(\Omega)$.
\end{remark}

\begin{lemma}\label{L:bary}
A Lipschitz map $F:\Omega\to\Omega$ preserves the Karcher mean if and only if it preserves the Karcher barycenter on $\Pro^1(\Omega)$.
\end{lemma}

\begin{proof}
Suppose $F$ preserves the Karcher mean.  By Remark \ref{R:pres} $F$ preserves the Karcher barycenter on $\Pro_0(\Omega)$.
By the density of $\Pro_0(\Omega)$ in $\Pro^1(\Omega)$, there exists a sequence $\{\mu_n\}\subseteq \Pro_0(\Omega)$ converging
to $\mu$.  Since $F$ is Lipschitz, the pushforward map $F_*:\Pro^1(\Omega)\to \Pro^1(\Omega)$ is continuou, so
$F_*(\mu_n)\to F_*(\mu)$.   By continuity of the barycentric map $\Lambda$, $\Lambda(F_*(\mu_n))\to \Lambda(F_*(\mu))$.
By hypothesis and continuity of $\Lambda$ and 
$$ F_*(\Lambda(\mu_n))=\Lambda(F_*(\mu_n))\to \Lambda(F_*(\mu))\mbox{ and } F_*(\Lambda(\mu_n))\to F_*(\Lambda(\mu)),$$
so $\Lambda(F_*(\mu))=F_*(\Lambda(\mu))$.  

In light of Remark \ref{R:pres}, the converse is immediate.
\end{proof}

\begin{theorem}\label{T:baryprops} Let $\mu\in\Pro^1(\Omega)$.  The barycenter map $\Lambda: \Pro^1(\Omega)\to \Omega$ satisfies the following properties.
\begin{itemize}
\item[(P1)] $\Lambda(t_*\mu)=t\Lambda(\mu)$, where $t_*\mu$ is the pushforward of $\mu$ under $a\mapsto ta$, $t>0$;
\item[(P2)] If $\mu(U)\leq \nu(U)$ for all open upper sets $U=\ua U=\{y:\exists x\in U,\, x\leq y\}$, then $\Lambda(\mu)\leq \Lambda(\nu)$;
\item[(P3)] The barycentric map $\Lambda$ is contractive from $\Pro^1(\Omega)$ with the Wasserstein metric to $\Omega$ with
the Thompson metric;
\item[(P4)] For $m\in \mathbb{A}$ invertible, $\Lambda(m^*\mu  m)=m^*\Lambda(\mu)m$, where $m^*\mu m $ is the pushforward of $\mu$ under $a\mapsto m^*am$;
\item[(P5)]  $\Lambda(\lambda\mu+(1-\lambda)\nu)\geq \lambda\Lambda(\mu)+(1-\lambda)\Lambda(\nu)$ for $0\leq \lambda\leq 1$;
\item[(P6)]  $\Lambda(J_*(\mu))=(\Lambda(\mu))^{-1}$, where $J(a)=a^{-1}$;
\item[(P7)]  $\mathcal{H}(\mu)\leq \Lambda(\mu)\leq \mathcal{A}(\mu)$, where $\mathcal{H}$ is the harmonic barycenter map and $\mathcal{A}$ is
the arithmetic barycenter map, and $\mu$ is chosen in the intersection of the three barycenter domains.

\end{itemize}
\end{theorem}

\begin{proof}  (P1) If we apply (P2) of Theorem \ref{T:props} to $\nu=(1/n)\sum_{i=1}^n \delta_{a_i}$, identified as customary with
$(\omega,(a_1,\ldots,a_n))$ where $\omega_i=1/n$ for $i=1,\ldots, n$, we obtain property (P1) for $\nu$.  Hence if we 
choose $\mu_n$ converging to $\mu$ with each $\mu_n\in\Pro_0(\Omega)$,  we obtain (P1) for $\mu$ in  the limit.

(P2)  The first part of (P2) defines a closed order on $\Pro^1(\Omega)$, called the stochastic order. see \cite[Proposition 5.4]{HLL}.   
For $\mu\leq \nu$, there exists sequences 
$\{mu_n\},\{\nu_n\}\subseteq \Pro_0(\Omega)$  such that $\mu_n\to \mu$, $\nu_n\to \nu$ and $\mu_n\leq \mu_n$ for each $n$; see 
Proposition 5.2 of \cite{La}, where is is also shown that the stochastic order restricted to $\Pro_0(\Omega)$ agrees with the order given in
(P4) of Theorem \ref{T:props} for some permutation of coordinates.  So using (P4) we obtain in the limit $\Lambda(\mu)\leq \Lambda(\nu)$,
since the stochastic order is closed.

(P3) This has already been alluded to in the previous material.  See \cite[Section 4]{LL16} for the original treatment.

(P4)  The map $F_m:\Omega\to\Omega$ defined by $F_m(a)=m^*am$ is an order isomorphism, hence an isometry for the
Thompson metric.  So (P4) follows from (P6) of Theorem \ref{T:props} and Lemma \ref{L:bary}.

(P5) This property follows from (P7) of Theorem \ref{T:props}, with the proof proceeding along the lines of the proof for (P1).

(P6) This property follows from (P8) with a proof similar to that for (P4). (In this case the \emph{anti-isomorphism} of inversion induces an isometry of the Thompson metric.)

(P7)  Let $\mu\in\Pro^1(\Omega,d)\cap \Pro^1(\Omega,\Vert\cdot\Vert)$.  Pick a sequence 
$\{\mu_n\}$ of finitely supported measures such that $\lim_n \mu_n=\mu$ both in 
$\mu\in\Pro^1(\Omega,d)$ and in $\Pro^1(\Omega,\Vert\cdot\Vert)$; see appendix.
It then follows from the continuity of the barycenter maps $\Lambda$ and $\mathcal{A}$,
the closeness of the stochastic order, and Theorem \ref{T:props} (P9) that 
$\Lambda(\mu)\leq \mathcal{A}(\mu)$.  

Note from (P9) of Theorem \ref{T:props} that $\mathcal{H}=J\mathcal{A}J_*$, where 
$J$ is the inversion map on $\Omega$, for all finitely supported  measures.  By continuity
of $\mathcal{H}$ and density of the set of finitely supported measures the equality holds
for all $\mu$ such that $J_*(\mu)\in\Pro^1(\Omega,\Vert\cot\Vert)$.  If also $\mu\in
\Pro^1(\Omega,d)$, we conclude from the previous paragraph that 
$$(\mathcal{H}(\mu))^{-1}=(J\mathcal{A}J_*(\mu))^{-1}=\mathcal{A}J_*(\mu)
\geq \Lambda (J_*\mu).$$. Taking the inverses of the extremes yields 
$\mathcal{H}(\mu)\leq (\Lambda(J_*\mu))^{-1}=\Lambda(\mu)$ by (P6).

\end{proof}

\section{Appendix}
In this section we develop a specific construction for approximating $\mu\in\Pro^1(X)$,
where $(X,\rho)$ is a metric space, by a probability measure with finite support. Our results are 
applied in the proof of (P7) of Theorem \ref{T:baryprops}.  

Let $C$ be the support of $\mu$, a closed separable set (Proposition \ref{P:support1}).
Fix $p\in X$.  Then $\int_X \rho(a,p)\,d\mu(a)=\int_C \rho(a,p)\,d\mu(a)<\infty$. 
%Let $R$ be the value of the integral.

Let $\mathcal{E}$ be a countable collection of Borel sets, each of  diameter less then $\ve$, 
that cover $C$.  We exclude from $\mathcal{E}$ those that miss $C$, if any, and enumerate the remaining ones, obtaining $\{E_n\}$.  We refine this into a pairwise disjoint sequence $\{B_n\}$
of Borel sets that cover $C$ as follows.  Set $B_1=E_1\cap C$.  Inductively if $B_1,\ldots,B_n$
have been chosen, set $B_{n+1}=(E_m\cap C)\setminus\bigcup_{i=1}^n B_i$, where $m$
is chosen as the least index for which this set is nonempty.  In this way one
obtains a covering sequence $\{ B_n\}$ (possibly finite) of $C$ consisting 
of pairwise disjoint nonempty Borel subsets of $C$ each of diameter less than $\ve$.  

Let $A_n=\bigcup_{i=1}^n B_n$.  The characteristic functions $\chi_n$ of $A_n$ form an increasing sequence
of nonnegative functions with supremum $\chi_C$.  By Lebesgue"s dominated convergence theorem
$$\lim_n\int_X \rho(a,p) \chi_n(a)\,d\mu(a)=\int_X \rho(a,p)\chi_C(a)\,d\mu(a)=\int_X \rho(a,p)\,d\mu(a),$$
the last equality holding since $C$ is the support of $\mu$.  Thus there exists $N$ such that
$\int_{X\setminus A_n} \rho(a,p)\,d\mu(a)=\int_X \rho(a,p)\,d\mu(a)- \int_X \rho(a,p) \chi_n(a)\,d\mu(a)<\ve$ for $n\geq N$.  

Pick $x_n\in B_n$ for each $n$.   We define a probability measure $\mu_n$ with finite support
as follows.  Let $\tau_i=\mu(B_i)$ for $1\leq i\leq n$.  We define
$$\mu_n=\sum_{i=1}^n \tau_i\delta_{x_i}+ \Big(1-\sum_{i=1}^n\tau_i\Big)\delta_p.$$

\begin{lemma}\label{L:seq}
For $n\geq N$, $d_w(\mu,\mu_n) < 2\ve$.  
\end{lemma}

\begin{proof}
We define $f_n:X\to X$ by $f_n(x)=x_i$ if $x\in B_i$ for $1\leq i\leq n$, and $f_n(x)=p$ otherwise.  
Note that $f_n$ is Borel measurable. We define $\mu_n$ to be the pushforward measure of $\mu$
 by $f_n$, $\mu_n=(f_n)_*(\mu)$ and note that it has support $\{x_1,\ldots,x_n,p\}$. We define $F_n:X\to X\times X$ by $F_n(x)=(x,f_n(x))$ and we let
$\pi_n=(F_n)_*(\mu)$  be the pushforward measure under $F_n$.  We note $\pi$ is a coupling 
of $\mu$ and $\mu_n$.  Using the change of variables formula we compute for $n\geq N$

\begin{eqnarray*}
\int_{X\times X} \rho(x,y) d\pi(x,y)&=&\int_{X\times X} \rho(x,y) d(F_n)_*(\mu)(x,y)\\
&=&\int_X \rho(x,f_n(x))d\mu(x)\\
&=&  \sum_{i=1}^n\int_{B_i} \rho(x,x_i)d\mu(x)+\int_{X\setminus A_n} \rho(x,p)d\mu(x)\\
&\leq& \ve\mu(A_n)+\ve<2\ve.
\end{eqnarray*}
Hence $d_w(\mu,\mu_n)<2\ve$.
\end{proof}

\begin{corollary}\label{C:two}
Let $d$ and $\rho$ be metrics on $X$ that generate the same topology.  Given $\mu\in\Pro^1(X,d)\cap\Pro^1(X,\rho)$, there exists
a sequence of finitely supported measures that converges to $\mu$ in $\Pro^1(X,d)$ and $\Pro^1(X,\rho)$,
\end{corollary}

\begin{proof}  Let $\ve=1/2n$. For each point $x\in C$, the support of $\mu$, pick an open set $U_x$ containing $x$ such that
$U_x$ has diameter less than $\ve$ in both metrics.  In either metric $C$ is separable, so there exists a countable subcover
$\mathcal{E}$ made up of various $U_x$. Lemma \ref{L:seq} and the construction preceding it yield a finitely supported measure $\mu_n$
such that $d_w(\mu,\mu_n)<2\ve=1/n$.  This inequality holds both in $\Pro^1(X,d)$ and $\Pro^1(X,\rho)$ since members of the
countable cover had diameter less than $\ve$ with respect to both metrics. The sequence $\{\mu_n\}$ attained in this fashion
is the one desired.
\end{proof}

\end{document}